\documentclass{birkjour}
\usepackage{amsfonts, amsthm, amsmath, amssymb}
\usepackage{float}
\usepackage[colorlinks=true,
            linkcolor=black,
            citecolor=black,
            urlcolor=black]{hyperref}
 \newtheorem{thm}{Theorem}[section]
 
 \newtheorem{lem}[thm]{Lemma}
 \newtheorem{prop}[thm]{Proposition}
 \theoremstyle{definition}
 \newtheorem{defn}[thm]{Definition}
 \theoremstyle{remark}
 
 \newtheorem*{ex}{Example}
 \numberwithin{equation}{section}

\newcommand{\C}{\mathbb C}
\newcommand{\R}{\mathbb R}
\newcommand{\N}{\mathbb N}

\newcommand{\Rf}{\mathfrak{R}}
\newcommand{\A}{\mathcal A}

\newcommand{\Rr}{\mathcal R}
\newcommand{\Ll}{\mathcal L}

\newcommand{\Ss}{\mathcal{S}}

\renewcommand{\geq}{\geqslant}

\DeclareMathOperator{\dom}{dom}
\DeclareMathOperator{\ran}{ran}

\DeclareMathOperator{\mul}{mul}

\begin{document}

%-------------------------------------------------------------------------
% editorial commands: to be inserted by the editorial office
%
%\firstpage{1} \volume{228} \Copyrightyear{2004} \DOI{003-0001}
%\seriesextraline{This is the Concrete Title of this Book\br H.E. R and S.T.C. W, Eds.}
%\firstpage{1}

\title[Hermitian Pencils and their Representation]
 {Hermitian Pencils and their Representation in Krein Spaces}

%----------Author 1

\author{Rabeb~Aydi}

\address{%
University of Sfax, Department of Mathematics\\
BP1171, Sfax 3000\\
Tunisia}

\email{rabeb.aydi.2019@gmail.com}
\author{Omaima Kchaou}

\address{%
University of Sfax, Department of Mathematics\\
BP1171, Sfax 3000\\
Tunisia}
\email{kchaouomaima53@gmail.com}

%----------Author 2
\author{Carsten~Trunk}
\address{Institute of Mathematics, , TU Ilmenau\br
 Weimarer Straße 25, 98693 Ilmenau\br
Germany}
\email{carsten.trunk@tu-ilmenau.de}


\dedicatory{In tribute to Heinz Langer – unique personality; a pioneer of operator pencils and the prince of indefinite inner product spaces}

\begin{abstract}
Pencils of the form  $\mathcal{A}(\lambda)=\lambda E-A$ are studied, where 
$A$ and $E$ are bounded linear operators on a Hilbert space. 
Of interest are the spectral properties of $\mathcal{A}(\lambda)$. This is done via a  corresponding linear relation in a Krein space, which is given in range representation using the two operators $A$ and $E$.
Under some assumptions on $E$ and $A$, the linear relation in range representation is nonnegative or has finitely many negative squares. Then one uses  spectral properties of
linear relations and deduces spectral properties of the operator pencil
$\mathcal{A}(\lambda)=\lambda E-A$.
\end{abstract}
%%% ----------------------------------------------------------------------
\maketitle
%%% 
\section{Introduction}
Spectral theory of operator pencils is a rich field in mathematics with various applications. Many authors have made important contributions to the field, including Keldysh, Krein, Langer, and Markus. There exists a vast amount of literature on this topic, e.g., see \cite{K51, KL78II, M88, MV15} and the references therein. 
One of the main methods in the spectral theory of pencils is the method of linearization, which replaces the study of nth order pencils acting in the Hilbert space $X$ by the study of equivalent problems for a companion linear pencil acting in the space $X^n$, see \cite{M88}.

A new direction in the development of the spectral theory of pencils are the investigations of Krein and  Langer for a quadratic operator pencil, \cite{ALM01, KL78I, KL78II, L76}. A particularly notable aspect of Krein's and especially of Langer’s influence was bringing operator pencil theory into the setting of Krein and Pontryagin spaces, see, e.g., \cite{L67,L73,L74,L75,L76}. 

In this 
paper we follow this path and develop a spectral theory for first order pencils of the form  
 \begin{equation}\label{Rakete}
     \mathcal{A}(\lambda) := \lambda E-A,
 \end{equation}
 where the operators $E$ and $A$ are bounded on the Hilbert space $X$. 

Operator pencils of first order are frequently used to model differential
algebraic equations \begin{align}\label{Eqdiff}
    \frac{d}{dt} E x(t)= A x(t), \quad E x(0) = E x_0,
\end{align}
 which can be viewed as coupled equations of differential equations subject to linear constraints. Differential algebraic equations of the form \eqref{Eqdiff} arise naturally in numerous applications from mechanics and electrodynamics \cite{FY98, RT05, TW19}. In finite dimension (i.e., $E$ and $A$ are matrices) this is a very active field, we mention here only the monographs 
 \cite{KM06} and \cite{LMT13}.
 
 The main objective of this paper is to investigate the spectral properties of the operator pencils  $\mathcal{A}(\lambda)$.
This is achieved through a corresponding linear relation $\Ll$,
\begin{equation} \label{range-kernel}
	\mathcal{L} = \ran\begin{bmatrix}
			E\\ A
			\end{bmatrix}, % = \text{Ker} \begin{bmatrix}
			%G, F
			%\end{bmatrix}.
\end{equation} 
which is called the range representation. Here the linear relation $\Ll$ is understood as a subspace of the cartesian product $X\times X$ of the underlying Hilbert space $X$.

In the special case that $E$ is boundedly invertible one multiplies $(\ref{Eqdiff})$ from the right by $E^{-1}$ and obtains (by adding an initial value) a standard Cauchy problem with generator  $AE^{-1}$. 

If $E$ is no longer invertible, this procedure can be repeated, where the inverse and the
 multiplication are in the sense of linear relations
 (for details see Section~\ref{Prelim}). In particular, one easily computes \begin{align}
     A E^{-1} =\ran\begin{bmatrix}
			E\\ A \end{bmatrix}.
 \end{align} 

\noindent The spectrum of the pencil $\A$ is given by
\begin{equation*}
   \sigma(\A):= \{ \lambda \in \C \mid \lambda E-A \hbox{ is not bijective} \},
\end{equation*} 
and $\infty\in\sigma(\A) $ if $E$ is not boundedly invertible.
Here we will always assume that there exists a real $\lambda \notin \sigma(\A)$, which corresponds to so called regular pencils. We denote the resolvent set of $\A(\lambda)$ by $\rho(\A)$, defined as the complement of the spectrum $\sigma(\A)$ in $\C\cup \{\infty\}$.

The expression $\Ll-\lambda I$ which is usually abbreviated as $\Ll-\lambda$, with \begin{equation}\label{Def: L-lambda}
	\Ll - \lambda := \{\{x,y-\lambda x\}\mid \{x,y\}\in \Ll\}.
\end{equation}
The spectrum of a linear relation $\Ll$ is defined as 
$$
\sigma(\Ll):=\{\lambda \in \C  \mid \Ll -\lambda I \ \text{is not the graph of a bijective operator}\}.
$$
Furthermore, $\infty\in \sigma(\Ll) $ if $\Ll$ is not the graph of a bounded operator. The resolvent set, denoted by $\rho(\Ll)$, is the complement of the spectrum $\sigma(\Ll)$ in $\C\cup \{\infty\}$.

Combining \eqref{range-kernel} and \eqref{Def: L-lambda} yields the link between the operator pencil and its associated linear relation \cite{GT21},
\begin{equation}\label{Arnstadt}
    \sigma (\A)=\sigma (\Ll).
\end{equation}

In this paper, we address the question of what can be said about the spectrum of operator pencils $\A(\lambda)$ under additional assumptions on $E$ and $A$ (e.g., hermitian).
For this we introduce a Krein space structure on $X$ via 
 $$ 
 [x, y]=(Jx,y), \ \hbox{with} \ x, y \in X \ \hbox{and} \ \ J= (\lambda E-A)^{-1},  
 $$ 
 where $(\cdot, \cdot)$ stands for the inner product in the Hilbert space $X$.  The space $X$, endowed with
the indefinite inner product $[\cdot,\cdot]$ generated by $J$ as above, is called a Krein space, see \cite{B74, IKL}. 
As our first main result, we show that the corresponding linear relation $\Ll$ in range representation is selfadjoint with respect to $[\cdot, \cdot]$. 

 As the second main result we show that $AJE$ is nonnegative, if and only if $\Ll$ is nonnegative with respect to $[\cdot, \cdot]$. Thus, since every nonnegative selfadjoint relation in a Krein space has real spectrum, combined with \eqref{Arnstadt}, gives $\sigma(\A)\subset \R$. 

Next, we concentrate on a subclass of selfadjoint linear relations  with $\kappa$ negative squares.  In our last main result we show that $AJE$ has $\kappa$ negative eigenvalues, if and only if  the linear relation $\mathcal{L}$ has finitely many negative squares. Hence, the spectrum of $\A(\lambda)$ is real with the exception of finitely many non-real eigenvalues, which occur in complex conjugate pairs.

\section{Preliminaries}\label{Prelim}
Let $X$ be a Hilbert space with inner product $(\cdot, \cdot)$ and let $J$ be a self-adjoint operator in $X$  which satisfies
\begin{equation}\label{Numero}
0\in \rho(J).
\end{equation}
Endow the space $X$ with the Hermitian sesquilinear form 
\begin{equation}\label{NumeroUno}
[\cdot, \cdot]= (J\cdot, \cdot) 
\end{equation}
 and, using the spectral projections corresponding to the intervals $(-\infty,0)$ and $(0,\infty)$, one obtains the direct and $[\cdot, \cdot]$-orthogonal decomposition of $X$,
$$
X = X_-[\dotplus] X_+,
$$
which is a so-called fundamental decomposition and, hence,
$(X, [\cdot,\cdot])$ is a Krein space, see \cite{AI89,B74,G22}.
If the dimension of $X_-$ ($X_+$) is $\kappa$, we say that the rank of negativity (resp.\ positivity) equals $\kappa$. Such spaces are then called Pontrjagin spaces. For the basic theory for Krein- and Pontrjagin spaces we refer to the monographs \cite{AI72,B74,G22,IKL}.

%We consider a Krein space inner product on $X^2:=X \times X$, \begin{align*}
%    [\{x,y\}, \{z,t\}]= [x, z]+[y,t], \ \ \{x,y\}, \{z,t\} \in X^2. \end{align*}
A linear subspace $\Ll$ of $X^2$ will be called linear relation in $X$. A closed linear relation $\Ll$ in $X$ is a closed linear subspace of $X^2$. Moreover, a closed linear operator in $X$ is viewed as a closed linear relation via its graph in $X^2$. We recall some basic notions for linear relations. For further details, see for example \cite{A61,DS1,DS2, LT77}. The domain and the kernel of a linear relation $\Ll$ in $X$ is denoted  by $\dom\Ll$ and $\ker\Ll$, respectively, \begin{eqnarray*}
	\dom\Ll&:= &\{x\in X ~ |~ \{x,y\} \in \Ll ~ \text{for some} ~ y\in X\},\\
	\ker\Ll&:= &\{x\in X ~ |~ \{x,0\} \in \Ll\}. \end{eqnarray*}
The inverse $\Ll^{-1}$ of $\Ll$ always exists and is defined as
\begin{align*}
    \Ll^{-1} :=\{\{y,x\}~|~  \{x,y\} \in \Ll \}.
\end{align*}
Moreover, $\ran\Ll$ and $\mul\Ll$ denote the range and the multivalued part of $\Ll$, 
$$
    \ran\Ll:= \dom \Ll^{-1},\quad \mbox{and} \quad
	\mul\Ll:= \ker \Ll^{-1}.
$$
Consider the linear relations $\Ll$ and $\Ss$ in $X$, then their sum and product is defined as
\begin{eqnarray*}
    \Ll+ \Ss &:=&\{\{x,y_1+y_2\}~|~ \{x,y_1\} \in \Ll, \ \{x,y_2\} \in \Ss \}, \\
    \Ll \Ss &:=&\{\{x,z\}~|~  \{x,y\} \in \Ss, \ \{y,z\} \in \Ll \}.
\end{eqnarray*}
The point spectrum $\sigma_p(\Ll)$ is given by
\[
\sigma_p(\Ll):=\{\lambda\in\C~|~ \ker(\Ll - \lambda)\neq\{0\}\}.
\] In addition, $\infty\in\sigma_p(\Ll)$ if $\mul \Ll \neq \{0\}.$
The adjoint relation (with respect to $[\cdot, \cdot]$) of $\Ll$ is defined by
\begin{align*}
    \mathcal{L}^{+}:=\{\{x,y\}\in X^2~|~ [x,t]=[y,z], \ \ \forall\ \{z,t\} \in \mathcal{L} \}.
    \end{align*}
A linear relation $\Ll$ is said to be symmetric (selfadjoint) with respect to $[\cdot, \cdot]$, if $\Ll \subseteq \Ll^{+}$
($\Ll = \Ll^{+}$, respectively). Furthermore, a linear relation $\mathcal{L}$ in $X$ is called nonnegative if \begin{align*}
    [f',f]\geq 0, \ \ \{f,f'\}\in \mathcal{L}.
\end{align*}
 A closed symmetric relation $\Ll$ in the Krein space  $(X, [\cdot, \cdot])$ is said to have $\kappa$ negative squares, $\kappa \in \mathbb{N}_0,$ if the Hermitian form $\langle\cdot, \cdot\rangle $ on $\Ll$, defined by \begin{align*}
    \langle\{f,f'\}, \{g,g'\}\rangle  := [f,g']= [f',g], \ \ \ \  \{f,f'\} ,\ \{g,g'\}\in \Ll, 
\end{align*}
    has $\kappa$ negative squares, that is, there exists a $\kappa$-dimensional subspace $\mathcal{M}$ of the symmetric relation $\Ll$ such that  
    $$
    [f', f] <0 
    $$
    for all $\hat{f}=\{f,f'\}\in \mathcal{M}$, $\hat{f}\neq 0$, but no $\kappa+1$-dimensional subspace with this property, see \cite{BLT13}.

\section{Operator pencils and their range representations}
We consider the operator pencil defined as 
$$
\A(\lambda)= \lambda E-A, 
$$ 
where the operators $E$ and $A$ are bounded on the Hilbert space $X$ with $E=E^{*}$ and $A=A^{*}$.  Here we will always assume that there exists a real $\lambda \notin \sigma(\A)$. Hence,
$$
J:= (\lambda E-A)^{-1}
$$ 
together with the inner product $[\cdot,\cdot]$ in \eqref{NumeroUno} is defined. Then $(X, [\cdot,\cdot])$
is a Krein space and $\Ll^+$ is the Krein space adjoint of $\Ll$ (see Section \ref{Prelim}). The next lemma states some elementary properties.

\begin{lem} \label{Properties} 
Let $\mathcal{L}$  be a linear relation as defined in \eqref{range-kernel}. Then,
\begin{equation} \label{Krein Adjoint of L} 
\mathcal{L}^{+} = \{\{x,y\}\in X \times X~|~ (JA)^{*}x=(JE)^{*}y \}.
\end{equation} 
Moreover,  
\begin{align} \label{dom}
    \dom \mathcal{L}^{+}=\{x\in X~|~ (JA)^{*}x\in \ran (JE)^{*}\},
\end{align} 
\begin{align} \label{ran}
 \ran \mathcal{L}^{+}=\{y\in X~|~ (JE)^{*}y\in \ran (JA)^{*}\}.
\end{align}
\end{lem}

\begin{proof}
By definition, the linear relation $\Ll$ can be written as $$\mathcal{L} = \textup{ran}\begin{bmatrix}
			E\\ A
			\end{bmatrix}=\{\{
			Ez, Az\}
			 ~|~ \ z\in X\}.$$
Let $\{x,y\}\in \mathcal{L}^{+}$. Thus,  $[v,x]=[u,y]$ for all $\{u,v\} \in \mathcal{L}$ and hence  
\begin{eqnarray} \nonumber     
 [Az,x]=[Ez,y] &\Leftrightarrow &  (JAz,x)=( JEz,y), \ \ \forall \ z \in X \\[1ex]\label{Freitag}
 &\Leftrightarrow & ( z,(JA)^{*} x)= ( z,(JE)^{*} y), \ \ \forall \ z \in X.
\end{eqnarray} 
Therefore, \eqref{Krein Adjoint of L}  follows.
By $\eqref{Krein Adjoint of L}$, one immediately obtains \eqref{dom} and \eqref{ran}.
\end{proof}
 
\begin{lem}  \label{Range representation}
 Let $\mathcal{L}$  be a linear relation as defined in \eqref{range-kernel}. Then 
 \begin{align}\label{Equation 6}
       AE^{-1}= \mathcal{L}= \textup{ran}\begin{bmatrix}
			EJ\\ AJ
			\end{bmatrix}= \textup{ran}\begin{bmatrix}
			EJ\\ -I+\lambda EJ
			\end{bmatrix},
    \end{align} 
 Further, the following holds.
 \begin{enumerate}
    \item $\mathcal{L}$ is closed in $X\times X.$
    \item $\ker\Ll^+= (\ran JA) ^\perp$ and $\mul\Ll^+= (\ran JE) ^\perp.$
    \item  $\ker\Ll= (\ran JAJ) ^\perp$ and $\mul\Ll= (\ran JEJ) ^\perp.$
\end{enumerate} 
\end{lem}

\begin{proof}
 We first prove $(\ref{Equation 6})$. By definition, we have that
 \begin{eqnarray}  \nonumber
 AE^{-1}&=& \{\{x,z\} \in X \times X ~|~ \{x,y\}\in E^{-1}, ~ \{y,z\} \in A \} \\[1ex]\label{Sylvester}
 &=& \{\{x,z\} \in X \times X ~|~  x=E y, ~ z= Ay, \ y \in X \} \\[1ex]\nonumber
 &=&  \textup{ran}\begin{bmatrix}
			E\\ A
			\end{bmatrix} = \mathcal{L}.
 \end{eqnarray}
 A substitution $x=(\lambda E-A)^{-1}y,$ $y\in X$, gives
  \begin{eqnarray*}\mathcal{L}&=& \{\{E(\lambda E-A)^{-1}y,A(\lambda E-A)^{-1}y\} ~|~  y\in X\} .
 \end{eqnarray*} 
 Obviously, one has $AJ=(A-\lambda E+\lambda E)J=-I+\lambda EJ$, so that  
 $$
        \mathcal{L}=  \textup{ran}\begin{bmatrix}
			EJ\\ -I+\lambda EJ
			\end{bmatrix}.
   $$ 
We now prove 1. Let 
$$
\{x_{n},z_{n}\}\in \textup{ran}
\begin{bmatrix}
			EJ\\ AJ
			\end{bmatrix}
$$
such that $x_{n}\rightarrow x$, $z_{n}\rightarrow z$, $n\rightarrow \infty$. It follows from \eqref{Sylvester} that $\{x_{n},z_{n}\}\in AJ(EJ)^{-1}$ and hence there exists $y_{n}\in X$ such that 
$$
\{x_{n},y_{n}\}\in (EJ)^{-1} \ \text{and}\ \{y_{n},z_{n}\}\in AJ.
$$
Then it is clear that 
$$
x_{n}=E Jy_{n}\quad  \mbox{and}\quad  z_{n}=A Jy_{n}=-y_{n}+\lambda EJy_{n}.
$$
We deduce that $y_{n}\rightarrow y$, $n\rightarrow \infty$, for some $y\in X$ with $y=\lambda x-z$.
The boundedness of $EJ$, implies  $\{y,EJy\}\in EJ,$ hence with $x=EJy$, we obtain $\{x,y\}\in(EJ)^{-1}$. From the boundedness of $AJ$ and $\{y_{n},z_{n}\}\in AJ $, we have that $\{y,z\}\in AJ.$ Overall, we conclude $\{x,z\}\in AJ(EJ)^{-1}=\mathcal{L}$. We continue with the proof of 2. From \eqref{Krein Adjoint of L} we obtain $\mul\Ll^+= \ker (JE)^*$ and $\ker\Ll^+= \ker (JA)^* $. Thus,  $\mul\Ll^+= (\ran JE) ^\perp$ and $\ker\Ll^+= (\ran JA) ^\perp$. The last claim follows in a similar way as in $2$.
\end{proof}
\section{Jordan chains of linear relations and operator pencils}
We recall the definition of the root subspaces (or algebraic eigenspaces) of order $k\geq 0,$  $\Rr_{\lambda}^k(\Ll)$ of $\Ll$ for $\lambda\in \C\cup\{\infty\}$, which are linear subspaces of $X$ defined by
\begin{equation*}
\begin{aligned}
\Rr_{\lambda}^ k(\Ll):= \ker (\Ll-\lambda)^k, \  \ \Rr_{\lambda}(\Ll):= \bigcup _{j=1} ^{\infty} \Rr_{\lambda}^ j(\Ll), \\
\Rr_{\infty}^ k(\Ll):= \ker \Ll^{-k}, \  \ \Rr_{\infty}(\Ll):= \bigcup _{j=1} ^{\infty} \Rr_{\infty}^ j(\Ll).
\end{aligned}
\end{equation*}
The dimension of $\Rr_{\lambda}(\Ll)$ is called the algebraic multiplicity of the eigenvalue $\lambda.$ Let $\lambda \in \C,$ then $x\in \Rr_{\lambda}(\Ll)$ if and only if for some $n\in \N$ there exists a chain of elements of the form 
\begin{equation}\label{Jordan}
 (x_n,x_{n-1}+\lambda x_n), (x_{n-1},x_{n-2}+\lambda x_{n-1}), \ldots, (x_2,x_1+\lambda x_2), (x_1, \lambda x_1) \in \Ll, \end{equation} and such that $x=x_n,$ the endpoint of \eqref{Jordan}. The chain in \eqref{Jordan} is said to be Jordan chain for $\Ll$ corresponding to the eigenvalue $\lambda \in \C$. Likewise, $h\in \Rr_{\infty}(\Ll) $, if and only if for some $l\in \N$ there exists a chain of elements of the form \begin{equation}\label{Jordan1}
 (0,h_1), (h_1,h_2), \ldots, (h_{l-2},h_{l-1}), (h_{l-1}, h_l) \in \Ll, \end{equation} and such that $h=h_l,$ the endpoint of \eqref{Jordan1}. The chain in \eqref{Jordan1} is said to be Jordan chain for $\Ll$ corresponding to the eigenvalue $\infty.$
\begin{defn}
     Let $\Ll$ be a linear relation in a vector space $X$. The Weyr characteristic of $\Ll$ is defined by
$$\omega_k(\Ll,\lambda):= \dim \frac{\Rr_{\lambda}^ k(\Ll)}{\Rr_{\lambda}^ {k-1}(\Ll)}, \ k\geq1.$$ 
\end{defn}

For operator pencils  as in \eqref{Rakete},
$$
\A(\lambda)= \lambda E-A, 
$$
the point spectrum $\sigma_p(\A)$ is given by 
\begin{equation*}
	\sigma_p(\A):=\{\lambda\in\C~|~ \ker(\lambda E-A)\neq\{0\}\},
\end{equation*}
and $\infty\in\sigma_p(\A)$ if $\ker E \neq \{0\}$.
The Weyr characteristic for operator pencils at eigenvalues $\lambda \in \sigma_p(\A)$ is defined in terms of root subspaces generated by Jordan chains, see \cite{K51, M88}.

A finite ordered set of nonzero vectors  $(x_1, \ldots, x_k) \in X^k, \ k\geq 1$, is called a Jordan chain of length $k$ at $\lambda \in \C$ if 
\begin{equation}\label{JC_def}
    (A-\lambda E)x_1=0, (A-\lambda E)x_2=E x_1, \ \ldots,\ (A-\lambda E)x_k=E x_{k-1},
\end{equation}
and $(x_1,\ldots,x_k)$ is a Jordan chain of length $k$ at $\infty$ if 
\begin{equation}\label{JC_infty}
E x_1 =0, E x_2 = Ax_1,\ \ldots,\ Ex_k = Ax_{k-1}.
\end{equation}
With the above notion we introduce the \emph{root subspaces} of the operator pencil $\A(\lambda)=\lambda E-A$ as follows
\begin{align*}
\Rr_\lambda^k(E,A):=\{x\in X~|~ \text{$x=x_k$ fulfills \eqref{JC_def}}\},\quad \Rr_\lambda(E,A):=\bigcup_{j=1}^\infty\Rr^j_\lambda(E,A),\\
\Rr_\infty^k(E,A):=\{x\in X~|~ \text{$x=x_k$ fulfills \eqref{JC_infty}}\},\quad  \Rr_\infty(E,A):=\bigcup_{j=1}^\infty\Rf^j_\infty(E,A).
\end{align*}

\begin{defn}
For operator pencils $
\A(\lambda)= \lambda E-A
$   as in \eqref{Rakete} with
$\rho(\A) \neq \emptyset$,
the Weyr characteristic is defined by
$$
\omega_k(\A,\lambda):= \dim \frac{\Rr_{\lambda}^ k(E,A)}{\Rr_{\lambda}^ {k-1}(E,A)}, \ k\geq1.
$$ 
\end{defn}

As shown in \cite{GT21}, different kinds of spectra of operator pencils and their associated linear relations coincide and likewise the Weyr characteristics.  
We repeat the findings from  \cite{GT21} in the following 
theorem.
\begin{thm}\label{weyr}
Assume $\rho(\A)\neq \emptyset$.
Then for all $\lambda\in \C\cup \{\infty\}$ and for all $k\geq 1$
$$
\begin{array}{c}
     \sigma(\Ll)=\sigma(\A), \\[1ex]
     \sigma_p(\Ll)=\sigma_p(\A),\\[1ex]
     \omega_k(\A,\lambda)= \omega_k(\Ll,\lambda).
\end{array}
$$
\end{thm}

\section{Hermitian pencils and selfadjoint relations in Krein space}
Throughout this section, we assume that the resolvent set $\rho(\A)$ of the operator pencil $\A(\lambda)=\lambda E-A$ is nonempty and contains a real number $\lambda$. Again we consider the operator $J$ given by 
\begin{equation}\label{Tashkent}
J:= (\lambda E-A)^{-1},
\end{equation}
where $E$ and $A$ are bounded linear operators, with $E=E^{*}$ and $A=A^{*}$. The operator $\lambda E-A$
is selfadjoint and hence 
$J$ turns out to be selfadjoint operator
and defines the inner product $[\cdot,\cdot]$ in \eqref{NumeroUno} which turns $(X, [\cdot,\cdot])$
into a Krein space (see Section \ref{Prelim}).
 The following theorem shows that the linear relation $\mathcal{L}$ is selfadjoint in the Krein space $X$. 
 
\begin{thm} \label{Self-adjoint linear relation}
     Let $\mathcal{L}$  be a linear relation defined as in \eqref{range-kernel} and $\lambda \in \rho(\mathcal{L}) \cap\mathbb{R}$. Then  $\mathcal{L}$ is selfadjoint with respect to $[\cdot, \cdot]$.
\end{thm}
\begin{proof} As established in Lemma \ref{Properties} and using the selfadjointness of $E$ and $A$ together with the fact that $\lambda$ is real, it follows that $J$ is selfadjoint and hence we obtain with Lemma \ref{Properties}
\begin{equation} \label{Equivalence self-adjoint} 
\mathcal{L}^{+} = \{\{x,y\}\in X \times X~|~ AJx=EJy \}.
\end{equation}  It can be seen from $\eqref{Equivalence self-adjoint}$ that $$\mathcal{L}^{+}= \textup{ker} 
\begin{bmatrix}
			AJ, -EJ
			\end{bmatrix}.$$
 Clearly, $AJ=-I+\lambda EJ$. This yields
\begin{eqnarray*}
  \mathcal{L}^{+} &=&   
    \{\{x,y\}\in X \times X~|~ \ -x + \lambda EJx =EJy \}\\ &=&
     \{\{x,y\}\in X \times X~|~  x=EJz,\ y=AJz\ \hbox{with}\ z=\lambda x-y\} \\&=&
     \{\{EJz,AJz\}, \ z\in X\}.  
\end{eqnarray*}
Therefore, \begin{equation}\label{Adjoint equality}
    \mathcal{L}^{+}=\textup{ran}\begin{bmatrix}
			EJ\\ AJ
			\end{bmatrix}. \end{equation}
 
\noindent Combining \eqref{Equation 6} and \eqref{Adjoint equality}, one concludes that $\mathcal{L}$ is selfadjoint.\end{proof}

\begin{prop}   Let $\mathcal{L}$  be a linear relation defined as in \eqref{range-kernel}. Then
 \begin{align*}
\mathcal{L}^{+}=J^{-1}\mathcal{L}^{*}J, \end{align*} where $\Ll^*$ is the Hilbert space adjoint with respect to $J$.\end{prop} 

\begin{proof} 
The adjoint of $\Ll$ in the Hilbert space $X$ is given by 
$$
\Ll^*=\{\{x,y\}\in X\times X~|~ A^*x=E^*y\}.
$$
Since $E$ and $A$ are selfadjoint, $\mathcal{L}^{*}= \textup{ker}\begin{bmatrix}
	A,- E
\end{bmatrix}$ and
    \begin{eqnarray*}
    \mathcal{L}^{*} J&=& \{\{x,z\} \in X \times X~|~  \{x,y\}\in J, \{y,z\} \in \mathcal{L}^{*}\} \\[1ex]
    &=& \{\{x,z\} \in X \times X ~|~ y=Jx,~ Ay=Ez\}\\[1ex]
    &=& \{\{x,z\} \in X \times X  ~|~  AJx=Ez \} \\ [1ex]
    &=&  \textup{ker} \begin{bmatrix}
			AJ,- E
			\end{bmatrix}.
 \end{eqnarray*} 
Therefore,
 \begin{eqnarray*}
 J^{-1}\mathcal{L}^{*} J&=& \{\{x,z\} \in X \times X ~|~  \{x,y\}\in \mathcal{L}^{*} J, \{y,z\} \in J^{-1}\} \\[1ex]
 &=& \{\{x,z\} \in X \times X ~|~ AJx=Ey ,~ y=Jz \} \\[1ex] 
 &=& \{\{x,z\} \in X \times X~|~  AJx=EJz\} \\[1ex]
 &=&  \textup{ker} 
 \begin{bmatrix}
			AJ, -EJ
			\end{bmatrix}. 
 \end{eqnarray*} As shown in Lemma \ref{Properties}, $\mathcal{L}^{+}=\textup{ker}\begin{bmatrix}
			AJ,- EJ
			\end{bmatrix}$  and the statement follows.
\end{proof}

The main results are now stated in the following theorem.
\begin{thm}\label{Linear relation properties}
Assume that $\lambda \in \rho(\mathcal{L}) \cap \mathbb R$  with $J= (\lambda E-A)^{-1}$. Then we have the following assertions.
\begin{enumerate}
\item The operator $AJE$ is selfadjoint.
    \item 
$AJE$ is a nonnegative operator, if and only if  $\mathcal{L}$ is nonnegative with respect to $[\cdot, \cdot]$. 
\item
The negative spectrum of $AJE$ consists only of finitely many eigenvalues with finite multiplicity, if and only if  the linear relation $\mathcal{L}$ has finitely many negative squares.
\end{enumerate} \end{thm}

\begin{proof}  
Since $E$ and $A$ are selfadjoint and $\lambda$ is real, it follows that $J$ is selfadjoint and hence $(AJE)^*=EJA$. Now, we need to show that $EJA=AJE$. It follows that 
$$
EJA=EJ(A-\lambda E+\lambda E) = -E+\lambda EJE.
$$
On the other hand, 
$$AJE=(A-\lambda E+\lambda E)JE = -E+\lambda E JE.
$$ 
Therefore, $EJA=AJE$ and hence $AJE$ is selfadjoint and the first claim is proven. 

Now, we prove the second claim.
Let $\{f,f'\} \in \mathcal{L}$. Then there exists $z\in X$ with $f=Ez$ and $f'=Az$. We have 
\begin{align*}
    [f,f']= [Ez,Az] = ( JEz, A z)= ( AJEz, z) \geq 0.
\end{align*} 
From this identity obviously the second and third claim follows.
\end{proof}

The following theorem is the main result of this paper. By $\C^+$ we denote the open upper half plane. Denote for an eigenvalue $\lambda$ of $\Ll$ the signature of the inner product $[\cdot, \cdot]$ on the algebraic eigenspace by $\{ \kappa_-(\lambda), \kappa_0(\lambda), \kappa_+(\lambda) \}.$
\begin{thm}\label{Theorem 4.5}
    Let $\A(\lambda) = \lambda E-A.$ Assume that $\rho(\A)\cap\R\neq \emptyset$. Then the following statements holds. 
    \begin{enumerate}
        \item If $AJE$ is nonnegative then the spectrum of $\A(\lambda)$ is real. 
        \item If the negative spectrum of $AJE$ consists only of finitely many eigenvalues with finite multiplicity, then the non-real spectrum of $\A(\lambda)$ consists of at most $\kappa$ pairs of complex conjugate points $\{\mu_i, \overline{\mu_i}\}$, $\mu_i\in \C^+$, of eigenvalues with finite dimensional algebraic eigenspaces and
        \begin{equation}
        \sum_{\lambda \in \sigma_p(\A)\cap (-\infty,0)}(\kappa_+(\lambda)+\kappa_0(\lambda))+ \sum_{\lambda \in \sigma_p(\A)\cap (0, \infty)}(\kappa_-(\lambda)+\kappa_0(\lambda))+ \sum_{i}\kappa_0(\mu_i)\leqslant \kappa,
    \end{equation} and equality holds if $0\notin \sigma_p(\A)$. Moreover, the  number of Jordan chains of $\A(\lambda)$ at different real nonzero eigenvalues of length greater than one is bounded by $\kappa.$ The length of each of these chains is at most $2\kappa +1$.
    \end{enumerate}
\end{thm}

The proof of Theorem \ref{Theorem 4.5} makes use of Lemma \ref{Lemma4} below which concerns certain spectral properties of selfadjoint linear relation with finitely many negative squares. 
The corresponding results for operators are well known
and are consequences of the general results in \cite{Lhabil, L82} (for a sketch of a proof we refer to \cite{BT07}; the case of one negative square was investigated by Jonas and Langer \cite{JL07}). However, for the sake of completeness we provide a short proof here
for  selfadjoint linear relation with finitely many negative squares.

\begin{lem}\label{Lemma4}
Let $\Ll$ be a selfadjoint linear relation in the Krein space $(X, [\cdot, \cdot])$ with $\kappa$ negative squares. Assume that $\rho(\Ll)\neq \emptyset$. Then the following holds. \begin{enumerate}
    \item The non-real spectrum of $\Ll$ consists of at most $\kappa$ pairs $\{\mu_i, \overline{\mu_i}\}$, $\mu_i \in \C^+$ of eigenvalues with finite-dimensional algebraic eigenspaces. Then 
    \begin{equation}\label{Fertig}
        \sum_{\lambda \in \sigma_p(\Ll)\cap (-\infty,0)}(\kappa_+(\lambda)+\kappa_0(\lambda))+ \sum_{\lambda \in \sigma_p(\Ll)\cap (0, \infty)}(\kappa_-(\lambda)+\kappa_0(\lambda))+ \sum_{i}\kappa_0(\mu_i)\leqslant \kappa,
    \end{equation}
    and equality holds if $0\notin \sigma_p(\Ll).$
    \item There are at most $\kappa$ different real nonzero eigenvalues of $\Ll$ with corresponding Jordan
 chains of length greater than one. The length of each of these chains is at most $2\kappa +1$.
\end{enumerate}
\end{lem}

\begin{proof}
We consider a linear space $\A=\dom\Ll$, with the Hermitian form 
        \begin{align}\label{Hermitian form}
            [f,g]_{\Ll}:= [f',g] \quad \mbox{for } \{f,f'\}, \{g,g'\} \in \Ll.
        \end{align} 
        As $\Ll=\Ll^+$ it follows 
        \begin{equation}\label{Ilmenau}
        [f',g]=[f,g'] \quad \mbox{for } \{f,f'\}, \{g,g'\} \in \Ll,
        \end{equation}
        which shows that $[\cdot , \cdot]_{\Ll}$ is well defined.
        
        Let $ \A^\circ$ be the isotropic part of $\dom \Ll$ with respect to
         $[\cdot , \cdot]_{\Ll}$. Let $f\in  \A^\circ$. Then we have for all $\{u,u'\} \in \Ll$ and together with \eqref{Ilmenau}
         $$
         0= [f,u]_{\Ll}=[f,u']=[0,u],
         $$
         hence $\{f,0\} \in \Ll^+= \Ll$ and $\A^\circ\subset \ker \Ll$. The other inclusion is obvious and thus
         $$
         \A^\circ =\ker \Ll.
         $$
Let $n$ and $m$ be in $\A^\circ$. Then $\{n,0\}, \{m,0\} \in \Ll$ and for $\{f,f'\}, \{g,g'\} \in \Ll=\Ll^+$
\begin{equation}\label{Ilmenau2}
 [f+n,g+m]_{\Ll}= [f',g+m] =  [f',g] + [f',m]= [f',g] +[f,0]
 = [f,g]_{\Ll}.
\end{equation}
Therefore, via \eqref{Ilmenau2}, one defines $[\cdot,\cdot]_{\Ll}$
on the quotient space $( \A/\A^\circ, [\cdot,\cdot]_{\Ll})$, which is a non degenerated inner product space.
Following \cite[Chapter 1, Theorem 2.5]{IKL}
it admits a completion to a Pontryagin space $(H, [\cdot,\cdot]_{\Ll})$  with rank of negativity of order $\kappa$. Define
$$
 \widehat{\Ll}:=\{\{f,f'\}\in \Ll\ |\ f' \in \dom\Ll\}\subset \A.
$$ 
For $\{f,f'\}, \{g,g'\} \in \widehat \Ll$ there exist
$f'', g''$ such that $\{f',f''\}, \{g',g''\} \in \Ll.$
We use \eqref{Hermitian form} for the pair $\{g,g'\}, \{f',f''\}$ and the pair $\{f,f'\}$, $\{g',g''\}$ and obtain
\begin{equation}\label{Winter}
[f',g]_{\Ll} = [f',g'] = [f,g']_{\Ll}.
\end{equation}
Taking equivalence classes, $\widehat{\Ll}$ induces in a natural way a linear relation $\Ll'$ in the quotient space $( \A/\A^\circ, [\cdot,\cdot]_{\Ll})$. By \eqref{Ilmenau2} and \eqref{Winter}, the closure $\overline{\Ll'}$ of $\Ll'$ in  the Pontryagin space $(H, [\cdot,\cdot]_{\Ll})$  is a closed symmetric linear relation.
 For all $z\in  \rho (\Ll)$ one has (see, e.g., \cite[Proposition 1]{GT21})
 $$
 \Ll=\ran
			\begin{bmatrix}
				(\Ll-z)^{-1}\\ I+z(\Ll-z)^{-1}
			\end{bmatrix}.
$$
Hence $\{(\Ll-z)^{-1}f,f+z(\Ll-z)^{-1}f\}\in \Ll$, for all $f\in\dom \Ll$, hence
$\{(\Ll-z)^{-1}f,f\}\in \widehat{\Ll}- z$ and $\ker ( \widehat{\Ll}- z) = \{0\}$ with
$$
\ran ( \widehat{\Ll}- z) = \dom(\Ll).
$$
It is straight-forward to see that
$$
z\in \rho(\overline{\Ll'}).
$$
Therefore, $\overline{\Ll'}$ is selfadjoint in the Pontryagin space $H$ with
 nonempty resolvent set. By \cite{DS2}, it follows that $\overline{\Ll'}$ is definitizable. 

Let $E_{\overline{\Ll'}}$ be the spectral function of $\overline{\Ll'}$, see \cite{KP15} and choose a bounded spectral set $[a,b] \subset (0,\infty)$ of $\overline{\Ll'}$ such that $[a,b]$ contains exactly one zero $\lambda.$ Thus, $(E_{\overline{\Ll'}}([a,b])X,[\cdot, \cdot])$ is a Pontryagin space and the rank of negativity of this space
is $\kappa_-(\lambda)+\kappa_0(\lambda)>0$. A similar statement holds for the negative zeros. The algebraic eigenspace corresponding to nonreal eigenvalues $\mu_i$ is neutral with respect to $[\cdot, \cdot]$ and the rank of negativity of $(E_{\overline{\Ll'}}(\{\mu_i, \overline{\mu}_i\})X,[\cdot, \cdot])$ is $\kappa_0(\mu_i)$. All these spaces are orthogonal to each other, hence its orthogonal sum is a gain in a Pontrygain space and a subspace of the underlying Pontryagin space $H$, which shows \eqref{Fertig}.

The second item follows from \eqref{Fertig} and the fact that in a Jordan chain of length $2\kappa +1$ at least $\kappa$ elements are neutral, which is shown exactly in the same way as for operators.
\end{proof}
%\begin{rem}
 %   Let $\Ll$ be a nonnegative linear relation in a Krein space $(X, [\cdot, \cdot])$, then $\sigma(\Ll) \subseteq \R.$ Moreover, by \cite[Proposition 5.2]{DS2} the Jordan chains of $\Ll$ at 0 are of length at most 2. 
%\end{rem}
%We now turn to the proof of Theorem \ref{Theorem 4.5}.
\begin{proof}[Proof of Theorem \ref{Theorem 4.5}]
If $AJE$ is nonnegative then by Theorem \ref{Linear relation properties} $\Ll$ is nonnegative with respect to $[\cdot, \cdot]$. Every nonnegative selfadjoint relation
(with non empty resolvent set) in a Krein space has real spectrum. Together with \eqref{Arnstadt} this implies that $\sigma(\A) \ \hbox{is real}$.

To prove the second claim assume that the negative spectrum of $AJE$ consists only of finitely many eigenvalues with finite multiplicity.  By Theorem \ref{weyr}, it follows that the Weyr characteristics of the linear relation $\Ll$ and the operator pencil $\A(\lambda)$ coincide. Hence, they share the same Jordan chains, including their number and lengths.  Thus, the claim follows from Lemma \ref{Lemma4}
and  Theorem \ref{weyr}.
\end{proof}

\section{A special case and an example}

\begin{thm} Let $A$ be a bounded selfadjoint operator in the Hilbert space $X$. If there exists $\lambda \in \rho(\mathcal{L}) \cap \mathbb R$ with $\lambda \geq 1$ and assume that $A$ is nonnegative, and assume $E=I+A$. Then, $\mathcal{L}$ is nonnegative. 
\end{thm}
\begin{proof}
    Let $E= I+A$ then, with $J=(\lambda I+ (\lambda -1)A)^{-1}$, we obtain for $z\in X$
\begin{eqnarray*}
    ( AJEz, z) &=& (A(\lambda I+ (\lambda -1)A)^{-1}(I+A)z, z) \\[1ex]
    &=& \int_{0}^\infty \frac{t}{\lambda + (\lambda -1)t}(1+t)d( E_tz,z),
\end{eqnarray*}
where $E_t$ denotes the spectral function of the nonnegative operator $A$.
Therefore, it is clear that $( AJEz, z)\geq 0$ and hence $\Ll$ is nonnegative.
\end{proof}

\begin{ex} Consider the
linear relation in range representation
    \[
\mathcal{L}=\ran \begin{bmatrix}
    E\\ A
\end{bmatrix},
\]
with $$E=\begin{bmatrix}
     1&0 & 0\\0&1&0 \\ 0& 0&0 
\end{bmatrix} \ \hbox{and} \ \ A=\begin{bmatrix}
    2&i&0\\-i&-1&0 \\ 0 & 0&1
\end{bmatrix}.$$ 
Therefore, both $E$ and $A$ are selfadjoint operators. Moreover, $$\lambda E-A=  \begin{bmatrix}
    \lambda -2 & -i  &0\\i& \lambda +1 & 0 \\ 0 & 0& -1
\end{bmatrix}. $$ Thus, $\det(\lambda E -A)=-\lambda^2+\lambda+3\neq 0$, if and only if $\lambda \neq \frac{1\pm \sqrt{13}}{2}$. Hence $$ J=(\lambda E-A)^{-1}= \frac{1}{-\lambda^2+\lambda+3} \begin{bmatrix}
    -\lambda-1 &-i & 0\\i& -\lambda +2 & 0 \\ 0 & 0 & \lambda^2 -\lambda -3
\end{bmatrix}.$$
From this, we see that 
$$ 
AJE=\begin{bmatrix}
    \frac{2\lambda+3}{\lambda^2-\lambda-3}&  \frac{i\lambda}{\lambda^2-\lambda-3} & 0\\\frac{-i\lambda}{\lambda^2-\lambda-3}& \frac{-\lambda+3}{\lambda^2-\lambda-3}&0 \\ 0 & 0& 0
\end{bmatrix},
$$
which is selfadjoint for $\lambda \neq \frac{1\pm \sqrt{13}}{2}$. The eigenvalues of the matrix $AJE$ are 
\begin{equation} \label{HelpMe}
\mu_0=0, \ \mu_+= \frac{B_\lambda}{2(\lambda^2-\lambda-3)} \ \hbox{and} \ \mu_-= \frac{S_\lambda}{2(\lambda^2-\lambda-3)}, 
\end{equation}
where 
$$
B_\lambda= (\lambda+6)+\sqrt{13} |\lambda|\quad \mbox{and}\quad S_\lambda=(\lambda+6)-\sqrt{13} |\lambda|.
$$
Note that $B_\lambda$ is positive for all real $\lambda$, whereas $S_\lambda$ changes its sign twice, see the Table
\ref{tab:sign-A-B} below.
\begin{table}[ht] \label{EINS}
\centering
\renewcommand{\arraystretch}{1.4} % increase row height
\caption{Sign analysis of $B_\lambda$ and $S_\lambda$}
\begin{tabular}{lcc}
\hline
  & \textbf{Sign of $B_\lambda$} & \textbf{Sign of $S_\lambda$} \\
\hline
 $\lambda \in \left(-\infty,\frac{-6}{1+\sqrt{13}}\right)$ & $+$ & $-$ \\[1ex]
  $\lambda \in \left(\frac{-6}{1+\sqrt{13}},
 \frac{6}{\sqrt{13}-1}\right) $ & $+$ & $+$ \\[1ex]
 $\lambda \in\left(\frac{6}{\sqrt{13}-1},\infty\right)$ & $+$ & $-$ \\
\hline
\end{tabular} 
\label{tab:sign-A-B}
\end{table}
Finally, the denominator in  \eqref{HelpMe} changes its sign in 
$$
\frac{1-\sqrt{13}}{2} = \frac{-6}{1+\sqrt{13}} \quad \mbox{and in}\quad  
\frac{1+\sqrt{13}}{2} = \frac{6}{\sqrt{13}-1}.
$$

\begin{table}[H]
\centering
\renewcommand{\arraystretch}{1.4}
\caption{Sign analysis of the eigenvalues $\mu_+$ and $\mu_-$ } \vspace{2mm}
\begin{tabular}{lcc}
\hline
 & \textbf{Sign of $\mu_+$} & \textbf{Sign of $\mu_-$} \\
\hline
 $\lambda \in \left(-\infty,\frac{-6}{1+\sqrt{13}}\right)$  & $+$ & $-$ \\[1ex]
$\lambda \in \left(\frac{-6}{1+\sqrt{13}},
 \frac{6}{\sqrt{13}-1}\right) $   & $-$ & $-$ \\[1ex]
$\lambda \in\left(\frac{6}{\sqrt{13}-1},\infty\right)$  
& $+$ & $-$  \\
\hline
\end{tabular}
\label{tab:sign-L}
\end{table}

From  Table \ref{tab:sign-L}, we see that if $\lambda \notin \left(\frac{-6}{1+\sqrt{13}},
 \frac{6}{\sqrt{13}-1}\right)$ then $AJE$ has one negative square. Otherwise, for $\lambda \in \left(\frac{-6}{1+\sqrt{13}},
 \frac{6}{\sqrt{13}-1}\right)$, the operator $AJE$ is nonpositive and by Theorem \ref{Theorem 4.5} (applied to 
 $-\A(\lambda)$ and $-AJE$) we have
 $$
 \sigma(\A) \mbox{ is real.}
 $$
\end{ex}

\subsection*{Data Availability Statement}
 No data were collected, generated or consulted
in connection with this research.

\subsection*{Declarations}

\subsection*{Conflicts of Interest}
 The authors have no Conflict of interest to declare that
are relevant to the content of this article.


\begin{thebibliography}{1}
\bibitem{ALM01} V.\ Adamyan, H.\ Langer, and M.\ M\"{o}ller, \textit{Compact perturbation of definite type spectra of self-adjoint quadratic operator pencils.} Integral Equations Operator Theory, \textbf{39} (2001), no. 2, 127--152.

\bibitem{A61} R. Arens, \textit{Operational calculus of linear relations.} Pac.\ J.\ Math. \textbf{11} (1961), 9--23.

\bibitem{AI89}
T.Y. Azizov and I.S. Iokhvidov, \textit{Linear operators in Space with an Indefinite Metric.} Chichester: John Wiley \& Sons Ltd., 1989.

\bibitem{AI72}
T.Y.\ Azizov and I.S.\ Iokhvidov, \textit{Linear operators in Hilbert spaces with $G$-metric.} Russ.\ Math.\ Surv.\ \textbf{26} (1972), 45--97.

\bibitem{BLT13} J. Behrndt, A. Luger, and  C. Trunk, \textit{On the negative squares of a class of self-adjoint extensions in Krein spaces.} Math.\ Nachr.\ \textbf{286} (2013), 118--148.

\bibitem{BT07} J. Behrndt and  C. Trunk,
\textit{On the negative squares of indefinite Sturm–Liouville
operators.} J.\ Differential Equations \textbf{238} (2007), 491–519.

%\bibitem{BP18}  J.  Behrndt and F.  Philipp, \textit{Finite rank perturbations in Pontryagin spaces and a Sturm-Liouville problem with $\lambda$-rational boundary conditions.} Indefinite Inner Product Spaces, Schur Analysis, and Differential Equations: A Volume Dedicated to Heinz Langer (2018), 163--189.

\bibitem{B74} J. Bognar, \textit{Indefinite Inner Product Spaces.} Springer, (1974).

%\bibitem{D98} V. Derkach, \textit{On Krein space symmetric linear relations with gaps}. Methods Funct. Anal. Topology \textbf{4} (1998), no. 2, 16--40.

\bibitem{DS1}
A. Dijksma and H.S.V. de Snoo,  \textit{Symmetric and selfadjoint relations in Krein Spaces I.} Operator Theory:
Advances and Applications \textbf{24} (1987), Birk\"{a}user Verlag Basel, 145--166.

\bibitem{DS2}
A. Dijksma and H.S.V. de Snoo,  \textit{Symmetric and selfadjoint relations in Krein Spaces II.} Ann.\ Acad.\ Sci.\ Fenn.\ Math.\ \textbf{12} (1987), 199--216.

\bibitem{FY98} A. Favini and A. Yagi. \textit{Degenerate Differential Equations in Banach Spaces.} CRC Press, 1998.

\bibitem{GT21} H. Gernandt and C. Trunk,  \textit{The spectrum and the Weyr characteristics of operator pencils and linear relations.} Syphax Journal of Mathematics: Nonlinear Analysis, Operator and Systems \textbf{1} (2021), 73--89.

\bibitem{G22}
A. Gheondea,  \textit{An Indefinite Excursion in Operator Theory: Geometric and Spectral Treks
in Krein Spaces.} Cambridge University Press, (2022).

\bibitem{IKL}
 I.S. Iohvidov, M.G. Krein, and H. Langer,  \textit{
 Introduction to the Spectral Theory of Operators in Spaces with an Indefinite Metric.} Akademie-Verlag, Berlin, (1982).


\bibitem{JL07}
P. Jonas and H. Langer.  \textit{On the spectrum of the self-adjoint extensions of a nonnegative linear relation of defect one in a Krein space. Operator theory in inner product spaces.} Oper.\ Theory Adv.\ Appl.\ Birkhäuser, Basel, \textbf{175} (2007), 121--158.

\bibitem{KP15}
M.\ Kaltenb\"{a}ck and R.\ Pruckner, 
\textit{Functional calculus for definitizable self-adjoint linear relations on {Krein} spaces},
Integral Equations Oper. Theory \textbf{83},
(2015), no.\ 4, 451--482.

%\bibitem{K95} T. Kato,  \textit{Perturbation theory for linear operators.} Classics in Mathematics, Springer, Berlin, 1995, Reprint of the 1980 edition.

\bibitem{K51}
M.V.\ Keldysh, \textit{On the eigenvalues and eigenfanctions of certain classes of nonselfadjoint equations.} Dokl. Akad. Nauk SSSR77 (1951), 11--14.



\bibitem{KM06}
P.~Kunkel and V.L.~Mehrmann,
\textit{Differential-Algebraic Equations. Analysis and Numerical Solution}, EMS Publishing House, Z{\"u}rich, Switzerland, 2006.

\bibitem{LMT13}
R.~Lamour, R.~M{\"a}rz, and C.~Tischendorf.
\textit{Differential Algebraic Equations: A Projector Based Analysis}, vol.~1 of {\em Differential-Algebraic Equations Forum}, Springer-Verlag, Heidelberg-Berlin, 2013.


\bibitem{KL78I} M.G.\ Krein and H. Langer, \textit{On some mathematical principles in the linear theory of damped oscillations of continua. I.} Translated from the Russian by R. Troelstra. Integral Equations Operator Theory \textbf{1} (1978), no. 3, 364--399.

\bibitem{KL78II} M.G.\ Krein and H.\ Langer, \textit{On some mathematical principles in the linear theory of damped oscillations of continua. II.}  Translated from the Russian by R. Troelstra. Integral Equations Operator Theory \textbf{1} (1978), no. 4, 539--566.

\bibitem{Lhabil} H.~Langer, 
\textit{Spektraltheorie linearer Operatoren in J-R\"{a}umen und 
einige Anwendungen auf die Schar $L(\lambda) = \lambda^2 + \lambda B + C$},
Habilitationsschrift, Technische Universit\"{a}t Dresden, 1965.

\bibitem{L67} H. Langer,
\textit{\"{U}ber stark ged\"{a}mpfte Scharen im Hilbertraum},
J. Math. Mech.\ \textbf{17} (1968), 685--705.

\bibitem{L73} H.\ Langer, \textit{\"{U}ber eine Klasse polynomialer Scharen selbstadjungierter Operatoren im Hilbertraum. II.},
J.\ Funct.\ Anal.\ \textbf{12} (1973), 13--29.

\bibitem{L74} H.\ Langer, 
\textit{\"{U}ber eine Klasse polynomialer Scharen selbstadjungierter Operatoren im Hilbertraum}, 
J.\ Funct.\ Anal.\ \textbf{16} (1974), 221--234.

\bibitem{L75} H.\ Langer,
\textit{Zur Spektraltheorie polynomialer Scharen selbstadjungierter Operatoren},
Math.\ Nachr.\ \textbf{65} (1975), 301--319.

\bibitem{L76} H.\ Langer, \textit{Factorization of operator pencils}, Acta Sci. Math.\ \textbf{38} (1976), 83--96.

\bibitem{L82} H. Langer,  \textit{Spectral functions of definitizable operators in Krein spaces.} Lect. Notes Math.\ \textbf{948} (1982), 1--46.

\bibitem{LNT} H.\ Langer, B.\ Najman, and C.\ Tretter, \textit{Spectral theory of the Klein-Gordon equation in Krein spaces}, Proc.\ Edinb.\ Math.\ Soc.\ \textbf{51} (2008), no. 3, 711--750.

\bibitem{LT77} H. Langer and B. Textorius, \textit{ On generalized resolvents and $Q$-functions of symmetric linear relations (subspaces) in Hilbert space.} Pacific J.\ Math., \textbf{72} (1977), no. 1, 135--165.


\bibitem{M88} A.S.\ Markus, \textit{Introduction to the Spectral Theory of Operator Polynomials.} AMS Trans. Monographs, Providence, RI, 1988.

\bibitem{MV15}
M. Möller and V. Pivovarchik. \textit{Spectral theory of operator pencils, Hermite-Biehler functions, and their applications. Operator Theory.} Advances and Applications, \textbf{246} Birkhäuser/Springer, Cham, (2015).

\bibitem{RT05} T.\ Reis and C.\ Tischendorf, \textit{Frequency domain methods and decoupling of linear infinite dimensional differential algebraic systems.} J. Evol. Equ.\ \textbf{5} (2005), no. 3, 357--385.

\bibitem{TW19}
S. Trostorff and M. Waurick, \textit{On differential-algebraic equations in infinite dimensions}, J.\ Differ. Equations \textbf{266} (2019), 526--561.
\end{thebibliography}
\end{document}